\documentclass[11pt,a4paper]{amsart}

\usepackage[left=3cm,right=3cm,top=3cm,bottom=2cm,foot=1.7cm]{geometry}
\usepackage{amsmath,amssymb,amsthm,mathrsfs}
\usepackage{amsaddr}

\usepackage{enumerate}
\newtheorem{theorem}{Theorem}
\newtheorem{lemma}{Lemma}
\newtheorem{proposition}[lemma]{Proposition}
\theoremstyle{definition}
\newtheorem{definition}{Definition}
\newtheorem*{remark}{Remark}

\newcommand{\mc}{\mathcal}
\newcommand{\ms}{\mathscr}

\newcommand{\mb}{\mathbb}
\newcommand{\mr}{\mathrm}

\title[On the dense Preferential Attachment Graph models]{On the dense Preferential Attachment Graph models and their graphon induced counterpart}

\author{\'Agnes Backhausz}
\address{E\"otv\"os Lor\'and University and MTA Alfr\'ed R\'enyi Institute of Mathematics\\P\'azm\'any P\'eter s\'et\'any 1/c, H-1117, Budapest, Hungary}
\email{agnes@math.elte.hu}

\author{D\'avid Kunszenti-Kov\'acs}
\address{MTA Alfr\'ed R\'enyi Institute of Mathematics\\P.O. Box 127, H-1364 Budapest, Hungary}
\email{daku@renyi.hu}

\keywords{dense graph limits, P\'olya urn processes, cut norm, jumble norm}
\subjclass[2010]{Primary: 05C80}

\date{\today}
\begin{document}
\maketitle
\begin{abstract}
Letting $\mc{M}$ denote the space of finite measures on $\mb{N}$, and $\mu_\lambda\in\mc{M}$ denote the Poisson distribution with parameter $\lambda$, the function $W:[0,1]^2\to\mc{M}$ given by
\[
W(x,y)=\mu_{c\log x\log y}
\]
is called the PAG graphon with density $c$. It is known that this is the limit, in the multigraph homomorphism sense, of the dense Preferential Attachment Graph (PAG) model with edge density $c$. This graphon can then in turn be used to generate the so-called W-random graphs in a natural way.\\
The aim of this paper is to compare the dense PAG model with the W-random graph model obtained from the corresponding graphon. Motivated by the multigraph limit theory, we investigate the expected jumble norm distance of the two models in terms on the number of vertices $n$. We present a coupling for which the expectation can be bounded from above by $O(\log^2 n\cdot n^{-1/3})$,
 and provide a universal lower bound that is coupling independent, but with a worse exponent.
\end{abstract}

\section{Introduction}

Preferential attachment graphs (PAGs) form a group of random growing graph models that have been studied for a long time \cite{barabasi, durrett, frieze}. The main motivation is  modelling randomly evolving large real-world networks, like online and offline social networks, the internet, or biological networks (e.g.\ protein-protein interactions). The basic PAG models have been extended by various features, for example duplication steps, weighted edges, vertices with random fitness.  The study of this wide family of models provided information about several phenomena in real-world networks (asymptotic degree distribution, clustering, relation of local and global properties, epidemic spread). The limiting behaviour of PAG models has also been investigated from various points of view, depending somewhat on the edge density along the graph sequences. For instance, in \cite{BBCS}, N.\ Berger, C.\ Borgs, J.\ T.\ Chayes and A.\ Saberi consider a sparse version of the process, with a linear number of edges compared to the number of vertices, and prove convergence in the sense of Benjamini--Schramm to a P\'olya point graph. A variation with added randomness is considered by R.\ Elwes in \cite{E1,E2}, where the preferential attachment model is amended in such a way that the number of edges added at each stage itself is a random variable, but in expectation still preserves a linear growth. The limit here is the infinite Rado graph, or a multigraph variant of the same, depending on whether multiple edges are allowed during the process.

At the dense end of the spectrum, C.\ Borgs, J.\ Chayes, L.\ Lov\'asz, V.\ S\'os and K.\ Vesztergombi considered in \cite{BCLSV} the case when the edge density along the sequence is essentially constant $c$ (i.e.\ the number of edges is approximately $cn^2/2$), under the convergence notion of injective graph densities. They showed that with probability 1 the graph sequence converges to the graphon $W:[0,1]^2\to\mb{R}$ given by $W(x,y)=c\ln x\ln y$. Later, B.\ R\'ath and L.\ Szak\'acs considered in \cite{RS} convergence of a more general family of processes with respect to induced graph densities, showing that the limit object is a graphon that now takes Poisson distributions as values instead.

If instead of considering induced densities, we look for homomorphism densities, the limit object can be seen to be in some sense a combination of the two previously mentioned ones: we obtain a graphon with $W(x,y)$ being a Poisson distribution with parameter $c\ln x\ln y$ (i.e., the injective density limit is the first moment of the homomorphism density limit). Hence the corresponding graphs contain multiple edges, and the original notions for limits of simple graphs cannot be used any more. 
The paper \cite{KKLS} by K.-K., L.\ Lov\'asz and B.\ Szegedy provides a framework for handling homomorphism densities in the context of multigraphs, and makes use of the so-called \textit{jumble-norm} to measure distance between graphons. 

All of the papers \cite{BCLSV,RS,KKLS} also deal with $W$-random graph sequences induced by the limit objects $W$, and show that with probability 1, the resulting graph sequence converges to $W$ in the respective densities sense.
These $W$-random graph models are thus very similar to the classical graph sequences that gave rise to the limit $W$, but also exhibit some significant differences.

Our goal in this paper is to compare the $c$-dense preferential attachment graph model to its $W$-random counterpart, showing that with probability 1 they are close (but not too close) in the jumble distance. The idea of the proof of the main result is to define a family of random graph models (see Section \ref{randomgraph}), which connects the $W$-random graph and the PAG model, and which can be coupled (see Section \ref{coupling}) so that the pairwise jumble-norm distances are easier to bound. In the discussion part (Section \ref{discussion}), we point out some features of the $W$-random version that can make it more useful in certain applications.

\section{Terminology and main result}

We shall start by defining the distance notion between multigraphs that we intend to use in this paper. It may be defined more generally for graphons (which essentially are weighted graphs with vertex set $[0,1]$), but that shall not be needed here, and we refer to \cite{KKLS} for more details.

\begin{definition}\label{def:jumble}
Let $G$ and $H$ be two (multi-)graphs on the same vertex set $[n]:=\left\{1,\ldots, n\right\}$ for some positive integer $n$. Then we define their \emph{jumble norm distance} as
\[
d_{\boxtimes}(G, H)=\frac 1n\cdot\max_{S,T\subseteq [n]} \frac{1}{\sqrt{st}}\bigg|\sum_{i\in S, j\in T} U_{ij}-V_{ij}\bigg|,
\]
where $U_{ij}$ and $V_{ij}$ denote the multiplicity of edge $ij$ in $G$ and $H$, respectively.
\end{definition}

The cut norm distance $d_\square$ used in many other papers (see e.g. \cite{BCLSV} for details) differs from this in the factor $\frac{1}{\sqrt{st}}$ that is omitted there. As such, our current distance notion magnifies the differences that occur on small sets, and we clearly have $d_\boxtimes\geq d_\square$. Also the jumble norm distance can be considered as an $L^2$-version of the cut norm distance, since $\sqrt{st}$ corresponds to the $L^2$ norm of the characteristic function of the set $S\times T$.

Next, fix a positive parameter $c>0$. Let $\mc{M}$ denote the space of finite measures on $\mb{N}$, and $W:[0,1]^2\to\mc{M}$ be the function given by
\[
W(x,y)=\mu_{c\log x\log y},
\]
where $\mu_\lambda$ denotes the Poisson distribution with parameter $\lambda$. We want to define the notion of $W$-random (multi-)graphs. The essence of the two-step randomization is as follows. We consider the set $[0,1]$ as the vertex set of the infinite graph with ``adjacency function'' $W$, and sample a random spanned subgraph on $n$ vertices by choosing its vertices independently uniformly from $[0,1]$. After this first randomization, we obtain a ``graph'' on $n$ vertices where each ``edge'' is a Poisson distribution. To obtain a true multigraph, we then independently sample an edge multiplicity for each pair of vertices from the corresponding Poisson distribution. If we allow loops, this will correspond to the random graph $\mb{G}_W^{\circ}(n)$, whereas if loops are disallowed, we obtain the random graph $\mb{G}_{W}(n)$.

\begin{definition} \label{def:gwn} We choose independent exponential random variables $\xi_i$ with parameter $1$ for every $1\leq i \leq n$. For $i<j$, let $Y_{ij}$ be a Poisson random variable with parameter $c\xi_i\xi_j$. For every $i$, let $Y_{ii}$ be a Poisson random variable with parameter $c\xi_i^2/2$. Assume that all $Y_{ij}$s are conditionally independent with respect to the $\xi_i$s. We put $Y_{ij}$ edges between vertices $i$ and $j$ for every $1\leq i \leq j \leq n$. This yields a random multigraph $\mathbb G_{\rm W}^{\circ}(n)$.\\ 
If, compared to $\mathbb G_{\rm W}(n)$, we erase the loops, we obtain the random multigraph $\mathbb G_{\rm W}(n)$.
\end{definition}

\begin{remark}
Note that using exponential variables instead of the uniform $[0,1]$ valued ones is compensated by the loss of the $log$ in the parameter.
\end{remark}

These are the random models we wish to compare to the below version of the PAG model.

\begin{definition}
We assign an urn to each vertex, initially with one single ball in each of them. Then we run a P\'olya urn process for $\lfloor cn^2\rfloor$ steps. That is, for $t=1, 2, \ldots, \lfloor cn^2\rfloor$, at step $t$, we choose an urn, with probabilities proportional to the number of balls inside the urn, and put  a new ball into it (each random choice is conditionally independent from the previous steps, given the actual distribution of the balls). Finally, for $k=1, 2, \ldots, \big\lfloor\lfloor cn^2\rfloor/2\big\rfloor$, we add an edge between the vertices where the balls at step $t=2k-1$ and at step $t=2k$ have been placed. This yields the random multigraph $\mathbb G_{\mr{PAG}}(n)$; multiple edges and loops may occur.
\end{definition}

It was proved in \cite{KKLS} that with probability 1, the random graph $\mb{G}_6(n)$ converges with respect to multigraph homomorphism densities to the original function $W$. As mentioned in the introduction, this is also the limit object obtained when looking at the random graphs $\mb{G}_1(n)$ defined as the preferential attachment graph on $n$ vertices with $\lfloor cn^2 \rfloor$ edges. \\
Given that letting $n$ go to infinity, the two random sequences $\mb{G}_1(n)$ and $\mb{G}_6(n)$ tend to the same limit, it is natural to ask how close these two sequences are as a function of $n$.

Our main result is that under an appropriate coupling, we obtain a polynomial bound on the expected distance.

\begin{theorem}\label{thm:main} There exists a coupling for which for every $1<\alpha<2$ there exists  $K(\alpha)>0$ such that for every $n\geq 1$ we have
\[\mathbb E\big(d_{\boxtimes}\big(\mathbb G_{\mr{PAG}}(n), \mathbb G_W(n)\big)\big)\leq K(\alpha)\cdot \log^2 n\cdot n^{\beta},\]
where $\beta:=\max_{\alpha\in (1, 2)}\left\{\alpha-2,\frac{1-\alpha}{2},-1/2,  4-3\alpha\right\}$.
With this bound, the optimum value for $\alpha$ is $5/3$, yielding $\beta=-1/3$.
\end{theorem}

In the last section, we provide a universal, coupling-independent lower bound of $O(n^{-1})$. The exponents are far from each other, but the lower bound uses very little of the structure of the models, so there is room for improvement.

\section{Random graph models}

\label{randomgraph}
We define a family of random graph models such that the neighboring ones are easier to compare in the jumble norm, and the whole family connects the two models of Theorem \ref{thm:main}. In the next section we will also present possible couplings for these pairs of models, which provide a coupling satisfying the conditions of the theorem. A positive number $c>0$ will be a common parameter of all of the models, and it will be considered fixed for the rest of the paper. Model 1 will be a realization of $\mb{G}_{\mr{PAG}}(n)$, whilst models 6 and 7 will be realizations of $\mb{G}_{\rm W}^\circ(n)$ and $\mb{G}_{\rm W}(n)$, respectively.\\
The graphs will have $n$ vertices, labeled by $1, 2, \ldots, n$. The parameter $\alpha$ will be chosen later so that the bounds are the best possible available from our approach.

\subsection*{Model 1} We assign an urn to each vertex, initially with one single ball in each of them. Then we run a P\'olya urn process for $\lfloor cn^2\rfloor$ steps. That is, for $t=1, 2, \ldots, \lfloor cn^2\rfloor$, at step $t$, we choose an urn, with probabilities proportional to the number of balls inside the urn, and put  a new ball into it (each random choice is conditionally independent from the previous steps, given the actual distribution of the balls). Finally, for $k=1, 2, \ldots, \big\lfloor\lfloor cn^2\rfloor/2\big\rfloor$, we add an edge between the vertices where the balls at step $t=2k-1$ and at step $t=2k$ have been placed. We obtain a random multigraph $\mathbb G_1(n)$ this way; multiple edges and loops may occur.

\subsection*{Model 2} Fix $\alpha\geq 0$. Let $r'$ be a random variable with negative binomial distribution, with parameters $n$ and $p_{\alpha}=1-e^{-\frac{1}{n^{\alpha-1}}}$ (we mean the version of negative binomial distribution with possible values $n, n+1, \ldots$). Let $r=r'-n$; this has values $0, 1, \ldots$ (sometimes this distribution is called negative binomial). The urn process is the same as in model $1$ (independent of $r'$), but we add edges between vertices chosen at step $t=2k-1$ and at step $t=2k$ only for $k\geq r/2$ (if $r>cn^2$, then we get the empty graph). We obtain a random multigraph $\mathbb G_2(n, \alpha)$.

\subsection*{Model 3} Let $\alpha$ and $r$ be defined as in model $2$. For $t=1, 2, \ldots, r$, we run the P\'olya urn as before. Let $R_i^*$ be the proportion of the  balls in urn $i$ after $r$ steps (for $i=1, \ldots, n$). For $t=r+1, \ldots, \lfloor cn^2\rfloor$, independently at each step, we put a new ball in an urn chosen randomly according to the distribution $(R_i^*)$. That is, the probability that the ball at step $t$ falls into urn $i$ is $R_i^*$, for all $t=r+1, \ldots, \lfloor cn^2\rfloor$. Finally, for $k\geq r/2$, we add an edge between the vertices chosen at step $t=2k-1$ and at step $t=2k$. (If $r>cn^2$, we mean the empty graph.) We obtain $\mathbb G_3(n, \alpha)$ this way. 

\subsection*{Model 4} Let $\alpha, r$ and $R_i^*$ be defined as in model $3$. If $r>cn^2$, take the empty graph. Otherwise, for every pair $1\leq i<j\leq n$, we take a random variable $Z_{ij}$ with Poisson distribution of parameter $cn^2R_i^*R_j^*$. For every $1\leq i\leq n$, we take a random variable $Z_{ii}$ with Poisson distribution of parameter $cn^2(R_i^*)^2/2$. We assume that all $Z_{ij}$s are conditionally independent of each other, given the $R_i^*$s. Finally, we put $Z_{ij}$ edges between vertices $i$ and $j$ for every pair $1\leq i\leq j \leq n$. We obtain $\mathbb G_4(n, \alpha)$ this way.

\subsection*{Model 5} Given $n$ and $\alpha$, the model is the same as model $4$ except that $r$ is not included any more; the model is the same as the previous one in the non-empty case.  We obtain $\mathbb G_5(n, \alpha)$ this way. 

\subsection*{Model 6} We choose independent exponential random variables $\xi_i$ with parameter $1$ for every $1\leq i \leq n$. For $i<j$, let $Y_{ij}$ be a Poisson random variable with parameter $c\xi_i\xi_j$. For every $i$, let $Y_{ii}$ be a Poisson random variable with parameter $c\xi_i^2/2$. Assume that all $Y_{ij}$s are conditionally independent with respect to the $\xi_i$s. We put $Y_{ij}$ edges between vertices $i$ and $j$ for every $1\leq i \leq j \leq n$. We obtain a random multigraph $\mathbb G_6(n)$ this way. 

\subsection*{Model 7} For every $1\leq i < j \leq n$, let $Y_{ij}$ be defined as in model $6$. We add $Y_{ij}$ edges between vertices $i$ and $j$ for all these pairs, but there are no loops in this case. We obtain $\mathbb G_7(n)$ this way.

\section{Couplings}

\label{coupling}
In order to prove Theorem \ref{thm:main}, we need to construct a particular coupling for which the distance of $\mathbb G_{\rm PAG}$ and $\mathbb G_{\rm W}$ is smaller than the upper bound. We do this through a sequence of couplings between the consecutive pairs, with respect to the order of random graph models in the previous section. It will be easy to see that the coupling of the first one (which is a realization of $\mathbb G_{\rm PAG}$) and the last one (which is a realization of $\mathbb G_{\rm W}$) can be constructed following the same order. At each step, we can simply add a finite family of random variables to the probability space independently where necessary, and use the already existing random variables in the other cases.   

\subsection*{Coupling of model $1$ and model $2$} These two models can be coupled easily. Take a realization of model $1$, and delete the edges corresponding to steps $2k-1$ and $2k$ for $k<r/2$. That is, we do not add the edges in the first $r$ steps.
\begin{proposition}\label{prop:12} For all $\alpha>1$ there exists $K_{1,2}>0$ such that 
\[\mathbb E\big(d_{\boxtimes}\big(\mathbb G_1(n, \alpha), \mathbb G_2(n, \alpha)\big)\big)\leq K_{1,2}\cdot\log n \cdot n^{\alpha-2}  \qquad (n=1, 2, \ldots)\]
holds in the coupling given above.
\end{proposition}

\subsection*{Coupling of model $2$ and model $3$} We start from a realization of model 2. Let $R_{i,t}$ be the proportion of the balls in urn $i$ after $t$ steps. Then, for $t=r+1, \ldots, \lfloor cn^2\rfloor$, conditionally on the process in model 2 until $t-1$ steps, we choose a  coupling of the distributions given by $(R_{i,t-1})_{i=1}^n$ and $(R_{i}^*)_{i=1}^n$ which minimizes the probability of choosing different urns and which is conditionally independent from the couplings used in the previous steps (with respect to the evolution of the number of balls). After adding the edges, we get a realization of model 3, because the distributions are determined by $(R_i^*)_{i=1}^n$, and the steps are conditionally independent of each other (and there is no difference in the first $r$ steps).  

\begin{proposition}\label{prop:23} For all $\alpha>1$ there exists  $K_{2,3}>0$ such that for every $n\geq 1$ we have
\[\mathbb E\big(d_{\boxtimes}\big(\mathbb G_2(n, \alpha), \mathbb G_3(n, \alpha)\big)\big)\leq K_{2,3}\cdot\log^2 n\cdot \left(n^{1/2-\alpha/2}+n^{\alpha-2}\right)\]
in the coupling given above.
\end{proposition}

\subsection*{Coupling of model $3$ and model $4$} The negative binomial random variable $r$ is common in the two models, this is chosen first. If $r>cn^2$, then both models give the empty graph, so we assume the contrary, and construct the coupling given $r$. Notice that in model $3$, since all steps are independent and use the same probability distribution, the edges are chosen independently, with probabilities proportional to $2R_i^*R_j^*$ for $i\neq j$ and $(R_i^*)^2$ for loops.

We assign independent Poisson processes to each pair of vertices. For $1\leq i<j\leq n$, the rate of the process is $2R_i^*R_j^*$ for $(i,j)$, and for $1\leq i\leq n$, the rate is ${R_i^*}^2$ for $(i,i)$. We denote by $N_s^{(ij)}$ the number of events until time $s$ in the $(i,j)$ process ($s>0$). The sum of these processes is also a Poisson process; let $\tau$ be the time when the total number of events reaches $\lfloor(\lfloor cn^2\rfloor-r)/2\rfloor+1$. If we put $N_{\tau}^{(ij)}$ edges between $i$ and $j$ for all $1\leq i \leq j \leq n$, then we get model $3$, because all $\tau$ events are distributed among the pairs of vertices independently, with probabilities proportional to the rates. On the other hand, if we put $N_{cn^2/2}^{(ij)}$ edges between $i$ and $j$, then we get model $4$, as the number of edges between the pairs are independent Poisson random variables with the appropriate parameter. Hence this provides a coupling of the two models. 

\begin{proposition}\label{prop:34} For all $\alpha>1$ there exists  $K_{3,4}>0$ such that for every $n\geq 1$ we have
\[\mathbb E\big(d_{\boxtimes}\big(\mathbb G_3(n, \alpha), \mathbb G_4(n, \alpha)\big)\big)\leq K_{3,4}\cdot\log n\cdot n^{\alpha-2}\]
in the coupling given above.
\end{proposition}

\subsection*{Coupling of model $4$ and model $5$}

For $r\leq cn^2$, there is no difference between the two models. Whenever $r>cn^2$, the graph $G_4$ is the empty graph, so no coupling is needed.

\begin{proposition}\label{prop:45} For all $2>\alpha>1$ there exists  $K_{4,5}>0$ such that for every $n\geq 1$ we have
\[\mathbb E\big(d_{\boxtimes}\big(\mathbb G_4(n, \alpha), \mathbb G_5(n)\big)\big)\leq K_{4,5}\cdot n^{-10}\]
in the coupling given above.
\end{proposition}

\subsection*{Coupling of model $5$ and model $6$}

First, we wish to couple the exponential random variables $\xi_i$ with the variables $R_i^*$ from the P\'olya urn. The following representation of the urn process until $r$ steps and its connection to independent exponential random variables yields a natural way to do this. In addition, this lemma will be useful when comparing models $1$ and $2$ as well.

\begin{lemma}\label{lem:exp}Fix $\alpha>1$. Let $r$ be defined as in model $2$. Let $X_i^*$ be the number of balls in urn $i$ (for $1\leq i \leq n$) after $r$ steps (we continue the P\'olya urn process even if $r>cn^2$). Let $\xi_1, \ldots, \xi_n$ be independent random variables with exponential distribution of parameter $1$. We define 
\[C_i=\lceil \xi_i n^{\alpha-1}\rceil \qquad (i=1, \ldots, n).\]
Then $(X_1^*, \ldots, X_n^*)$ and $(C_1, \ldots, C_n)$ have the same joint distribution.
\end{lemma}
\begin{proof}After $r$ steps, the total number of balls is $r+n$; that is, $\sum_{i=1}^n X_i^*=r+n$.
As it is well known, by the interchangeability property of the chosen colors in the urn process, for every $s\geq n$ and $\sum_{i=1}^n k_i=s$  we have
\[\begin{split}\mathbb P&\left(X_1^*=k_1, \ldots, X_n^*=k_n\left|\sum_{i=1}^n X_i^*=s\right.\right)\\&=\binom{s}{k_1-1}\binom{s-k_1+1}{k_2-1}\ldots \binom{s-k_1-\ldots-k_{n-2}+n-2}{k_{n-1}-1}\cdot\frac{(k_1-1)!\ldots (k_n-1)!}{n(n+1)\ldots(n+s-1)}\\&=\frac{s!(n-1)!}{(n+s-1)!}=\binom{n+s-1}{n-1}^{-1}.\end{split}\]
On the other hand, for every $k\geq 0$ and $1\leq i \leq n$, the definition of $C_i$ implies that  
\begin{equation}\label{eq:ci}\mathbb P(C_i\geq k)=\mathbb P(\xi_in^{\alpha-1}> k-1)= \exp\left(-\frac{k-1}{n^{\alpha-1}}\right)=\bigg(\exp\left(-\frac{1}{n^{\alpha-1}}\right)\bigg)^{k-1}.\end{equation}
Hence  $C_i$ has geometric distribution of parameter $p_{\alpha}=1-e^{-\frac{1}{n^{\alpha-1}}}$ (where we mean the version with possible values $1, 2,\ldots$). The random variables $C_i$s are independent, thus $\sum_{i=1}^n C_i$ has the same negative binomial distribution as $r+n$. Hence  $\sum_{i=1}^n X_i^*$ and $\sum_{i=1}^n C_i$ have the same distribution.  In addition, the conditional distributions given the sum are also the same, because we have 
\begin{align*}\mathbb P(C_1=k_1, \ldots, C_n=k_n)=(1-p_{\alpha})^{k_1-1}p_{\alpha}\ldots(1-p_{\alpha})^{k_n-1}p_{\alpha}=p_{\alpha}^n(1-p_{\alpha})^{\sum_{i=1}^n k_i-n}.\end{align*}
This depends only on the sum of the $k_i$s, which implies that 
\[\mathbb P\left(C_1=k_1, \ldots, C_n=k_n\bigg|\sum_{i=1}^n C_i=s\right)=\binom{n+s-1}{n-1}^{-1},\]
just as we have seen in the previous case. 
\end{proof}

Recall that the $R_i^*$-s corresponded to the ratio of the colors in the urn after $r$ steps, and therefore the P\'olya urn model can be coupled to the family of random variables $(\xi_i)$ in such a way that
\[
R_i^*=\frac{\lceil \xi_i n^{\alpha-1}\rceil}{\sum_{j=1}^n \lceil \xi_j n^{\alpha-1}\rceil}=\frac{C_i}{\sum_{k=1}^n C_k}.
\]

Next we couple the Poisson random variables $Y_{ij}$ and $Z_{ij}$ for each pair $1\leq i\leq j\leq n$. We exploit the fact that the sum of two independent Poisson distributions is again a Poisson distribution whose parameter is the sum of the original parameters. Let $\mc{F}$ be the $\sigma$-algebra generated by the families $(\xi_i)$ and $(R_i^*)$. Conditioned on $\mc{F}$, the coupling is done so that for each pair $1\leq i< j\leq n$, we generate independent Poisson random variables $H_{ij}$ and  $H^*_{ij}$ of parameter $\mu_{ij}:=cn^2\min\{\xi_i\xi_j,R_i^*R_j^*\}$ and $\mu_{ij}^*:=cn^2\left|\xi_i\xi_j-R_i^*R_j^*\right|$ respectively, and set
\[
\begin{array}{lcr}
Y_{ij}:=H_{ij}+\mb{I}(\xi_i\xi_j<R_i^*R_j^*)H^*_{ij} & \mbox{ and } & Z_{ij}:=H_{ij}+\mb{I}(\xi_i\xi_j>R_i^*R_j^*)H^*_{ij}.
\end{array}
\]
For the variables $Y_{ii}, Z_{ii}$, the coupling is done similarly, with all parameters halved.

\begin{proposition}\label{prop:56} For all $\alpha>1$ there exists  $K_{5,6}>0$ such that for every $n\geq 1$ we have
\[\mathbb E\big(d_{\boxtimes}\big(\mathbb G_5(n, \alpha), \mathbb G_6(n)\big)\big)\leq K_{5,6}\cdot (\log n)^{1/2}\cdot \big(n^{-1/2}+n^{4-3\alpha}\big)\]
in the coupling given above.
\end{proposition}

\subsection*{Coupling of model $6$ and model $7$}

Generate $G_6$, then delete the loops. This yields the natural coupling between $G_6$ and $G_7$.
\begin{proposition}\label{prop:67} There exists  $K_{6,7}>0$ such that for every $n\geq 1$ we have
\[\mathbb E\big(d_{\boxtimes}\big(\mathbb G_6(n), \mathbb G_7(n)\big)\big)\leq K_{6,7}\cdot n^{-3/4}\]
in the coupling given above.
\end{proposition}

We also conclude that this sequence of couplings can be realized in a single probability space, if we start with an appropriate family of independent random variables. Thus we constructed a coupling of $\mathbb G_{\rm PAG}$ and $\mathbb G_{\rm W}$.

\section{Proofs}

\subsection*{Proof of Theorem 1}
The result follows from the triangle inequality and Propositions 1 through 6.\hfill $\square$

We shall therefore now turn our attention to proving the bounds connecting each pair of models.
Since the jumble norm distance is not always easy to work with, we shall make use of the following lemma.

\begin{lemma} \label{lem:sorosszeg}
Let $G$ and $H$ be two (undirected) multigraphs on the vertex set $\{1, 2, \ldots, n\}$. Let $U_{ij}$ be the number of edges between $i$ and $j$ in $G$, and $V_{ij}$ the same quantity in $H$.  Then the following holds:
\[d_{\boxtimes}(G, H)=\frac 1n\cdot\max_{S,T} \frac{1}{\sqrt{st}}\bigg|\sum_{i\in S, j\in T} U_{ij}-V_{ij}\bigg|\leq \frac 1n\cdot\max_{1\leq i \leq n}\ \sum_{j=1}^n |U_{ij}-V_{ij}|.\]
\end{lemma}
\begin{proof}Let  $\sigma_i=\sum_{j=1}^n |U_{ij}-V_{ij}|$. 
Notice that if  $|S|=s$, $|T|=t$, and $s\leq t$, then
\[\bigg|\sum_{i\in S, j\in T} U_{ij}-V_{ij}\bigg|\leq \sum_{i\in S, j\in T} |U_{ij}-V_{ij}|\leq \sum_{i\in S} \sigma_i\leq s \max_{1\leq i \leq n} \sigma_i.\]
Hence
\[\frac{1}{\sqrt{st}}\bigg|\sum_{i\in S, j\in T} U_{ij}-V_{ij}\bigg|\leq \frac{s\max_{1\leq i \leq n} \sigma_i}{\sqrt {st}}=\frac{\sqrt s}{\sqrt t}\max_{1\leq i \leq n} \sigma_i\leq \max_{1\leq i \leq n} \sigma_i,\]
as we assumed that $s\leq t$. In the reverse case $s\geq t$, we get the same with the bound $ \max_{1\leq j \leq n} \sum_{i=1}^n{|U_{ij}-V_{ij}|}$. Since $U_{ij}=U_{ji}$ and $V_{ij}=V_{ji}$, this is equal to the previous maximum. This finishes the proof.
\end{proof}

\subsection{Models $1$ and $2$}

\subsection*{Proof of Propositon $\ref{prop:12}$} Let $U_{ij}$ be the number of edges between $i$ and $j$ in model $1$, and $V_{ij}$ the number of edges between $i$ and $j$ in model $2$. By the definition of the coupling, $U_{ij}$ can never be smaller than $V_{ij}$. If $r<cn^2$, then $U_{ij}-V_{ij}$ is the number of edges added to model $1$ during the first $r$ steps. Therefore $\sum_{j=1}^n |U_{ij}-V_{ij}|$ is at most the number of steps in which urn $i$ was chosen during the first $r$ steps, which is $X_i^*-1$ (cf. Lemma \ref{lem:exp}). Even if $r\geq cn^2$, the sum $\sum_{j=1}^n |U_{ij}-V_{ij}|$ cannot be larger than $cn^2/2$, since there are no more edges in model $1$. By Lemma \ref{lem:sorosszeg} and Lemma \ref{lem:exp}, we obtain 
\[\mathbb E\big(d_{\boxtimes}\big(\mathbb G_1(n, \alpha), \mathbb G_2(n, \alpha)\big)\big)\leq \mathbb E\big(\min\big(\max_{1\leq i \leq n} X_i^*, cn^2\big)\big)=\mathbb E\big(\min\big(\max_{1\leq i \leq n} C_i, cn^2\big)\big).\]
Equation \eqref{eq:ci} implies 
\[\mathbb P\left(\max_{1\leq i \leq n} C_i>3\log n\cdot n^{\alpha-1}+1\right)\leq \sum_{i=1}^n  \mathbb P\left( C_i>3\log n\cdot n^{\alpha-1}+1\right)\leq ne^{-3\log n}=\frac{1}{n^2}.  \]
Hence the expectation of the minimum is at most $3\log n\cdot n^{\alpha-1}$ plus some constant depending only on $c$. This finishes the proof. \hfill $\square$

\subsection{Models $2$ and $3$}

The idea of the proof of Proposition \ref{prop:23} is to find the expected value of the maximum when all global random variables (like $r$) are close to their mean, and then use large deviation theorems to show that this is the case with high probability. Throughout this proof, the constant factor in the $O(\cdot)$ notation may depend only on $c$.

First we fix $1\leq i \leq n$. Let $X_{i,t}$ be the number of balls in urn $i$ after $t$ steps. Recall that $X_i^*$ denotes the number of balls in urn $i$ after $r$ steps. We define the proportions similarly (recall that the initial configuration consists of one ball at each urn): 
\[R_{i,t}=\frac{X_{i,t}}{t+n}; \qquad R_i^*=\frac{X_i^*}{r+n}.\] 

We will use an  application of de Finetti's theorem to the urn process $X_t$ (see e.g.\ Theorem 2.2.\ in \cite{pemantle}). The joint distribution of the urns chosen randomly  can be represented as follows. Let $p$ be a random variable with distribution $\mathrm{Beta}(1, n-1)$ (as there is a single ball in urn $i$ at the beginning and $n-1$ balls in the other urns). Then, conditionally on $p$, generate independent Bernoulli random variables taking value $1$ with probability $p$. This has the same distribution as the indicators of the steps when a new ball is placed to urn $i$. This representation has an immediate consequence on the maximum of the proportion.

\begin{lemma} \label{lem:maxrt}
\begin{enumerate}[(a)] \item Let $p$ be a random variable with distribution $\mathrm{Beta}(1, n-1)$ with $n\geq 1$. Then we have 
\begin{equation}\label{eq:kisp}\mathbb P\left(p>\frac {16}n \log n\right)\leq n^{-8}.\end{equation}
\item For every $1\leq i \leq n$ we have
\begin{equation}\label{eq:maxrt}
\mathbb P\left(\max_{n\leq t \leq cn^2} R_{i,t}>\frac{36}{n}\log n\right)\leq 2cn^{-6}.\end{equation}
\end{enumerate}
\end{lemma}

\begin{proof} $(a)$ By using that  $n-1\geq n/2$, we have 
\begin{equation*}\begin{split}
\mathbb P\left(p>\frac {16} n \log n \right)&=\int_{16\log n/n }^1 (n-1)(1-x)^{n-2}dx=\int_0^{1-16\log n/n} (n-1)x^{n-2}dx\\&=(1-16\log n/n)^{n-1}\leq \exp(-8 \log n)=n^{-8}.
\end{split}\end{equation*}
$(b)$ Using exponential Markov's inequality and part $(a)$, we have \begin{align*}\mathbb P&\left(R_{i,t}>\frac{36}{n}\log n\right)\leq\mathbb P\left(R_{i,t}>\frac {36}n \log n\bigg\vert p\leq \frac{16}{n}\log n\right)+n^{-8}\\
&=\mathbb P\left(X_{i,t}>\frac {36(t+n)}n \log n\bigg\vert p\leq \frac{16}{n}\log n\right)+n^{-8}\\
&\leq \frac{\mathbb E((1+(e-1)p)^t|p\leq\frac {16} n\log n)}{\exp(\log n\cdot 36 (t+n)/n)}+n^{-8}\leq \frac{\exp((e-1)t\cdot\frac {16} n \log n)}{\exp(\log n\cdot {36}(t+n)/n)}+n^{-8}\\
&\leq \exp\left(((e-1)\cdot 16-36)\frac t n \log n\right)+n^{-8}\leq \exp(-8\log n)+n^{-8}\leq 2n^{-8},\end{align*}
where we assumed that $t\geq n$. This immediately implies $(b)$.
\end{proof}

We will use the following lemma, which is based on a large deviation argument.

\begin{lemma}\label{lem:binom}
Fix integers  $m\geq n\geq 2$. 
Let $p$ be a random variable with distribution $\mathrm{Beta}(1, n-1)$. Let $\eta$ be a random variable whose conditional distribution with respect to $p$ is binomial with parameters $m$ and $p$. We define 
\[B_m=\bigg\{\frac{3600 \log  n}{m}< p< \frac{16\log n}{n}\bigg\}.\]
Then there exists $K_1>0$ such that
\[\mathbb P\left(\bigg\{|\eta-mp|\geq K_1\sqrt{\frac{m}{n}}\log n\bigg\}\cap B_m\right)=O(n^{-8}).\]
\end{lemma}
\begin{proof}
We will compare the difference $|\eta-mp|$ to the variance of the binomial distribution, given $p$. We start with
\begin{equation}\label{eq:py1}\begin{split}
\mathbb P\left(\bigg\{|\eta-mp|\geq K_1\sqrt{\frac{m}{n}}\log n\bigg\}\cap B_m\right)&\leq \mathbb P\left(\bigg\{|\eta-mp|> K\sqrt{mp(1-p)\log n}\bigg\}\cap B_m\right)\\&\quad+\mathbb P\left(K\sqrt{mp(1-p)\log n}> K_1\sqrt{\frac{m}{n}}\log n\right).
\end{split}\end{equation}
We will choose $K=6$ but keep writing $K$ for clarity. Since $B_m$ is measurable with respect to $p$, the first term is equal to
\begin{equation}\label{eq:epib}q_1=\mathbb E\left(\mathbb P\left(\bigg\{|\eta-mp|> K\sqrt{mp(1-p)\log n}\bigg\}\bigg|p\right)\cdot \mathbb I_{B_m}\right),\end{equation}
where $\mathbb I_{B_m}$ denotes the indicator function of the event $B_m$.  

We define $k=mp-K\sqrt{mp(1-p)\log n}$ and $k'=mp+K\sqrt{mp(1-p)\log n}$; then the first event in \eqref{eq:epib} is $\{\eta/m<k/m\}\cup \{\eta/m>k'/m\}$. It is clear that $k/m<p$ and $k'/m>p$; hence we can apply large deviation arguments. Furthermore, we have $k/m>0$ on the event $B_m$, as the following calculation shows.  
\[p>K^2\frac{\log n}{m}\Leftrightarrow \sqrt p>K \sqrt{\frac{\log n}{m}}\Rightarrow p>K\sqrt{\frac{p(1-p)\log n}{m}}.\]
We also need $k'/m<1$. That is, we have to check whether the following holds:
\begin{align*}mp+K\sqrt{mp(1-p)\log n}&<m; \\
K\sqrt{mp(1-p)\log n}&<m(1-p);\\
K\sqrt{p\log n}&<\sqrt{m(1-p)}.\end{align*}
Since we have $p<16\log n/n$ on $B_m$ and we assumed $m\geq n$, this holds for large enough $n$ (recall that $K=6$ does not depend on any of the parameters).

Hence we can apply the relative entropy version of the Chernoff bound for binomial distributions, conditionally with respect to $p$. We obtain 
\begin{align*}\mathbb P(\eta/m< k/m)&\leq \mathbb E\left(\exp\left(-m D\left(\frac km \bigg\| p\right)\right)\right); \\ \mathbb P(\eta/m> k'/m)&\leq \mathbb E\left(\exp\left(-m D\left(\frac km \bigg\| p\right)\right)\right),\end{align*}
where $D(a\| p)= a \log \frac ap+(1-a)\log \frac{1-a}{1-p}$.
We need the following quantities for the calculations.
\begin{align*}
\frac{k}{m}&=\frac{mp-K\sqrt{mp(1-p)\log n}}{m}=p-K\sqrt{\frac{p(1-p)}{m}\log n};\\
\frac{k}{mp}&=1-K\sqrt{\frac{1-p}{mp}\log n};\\
1-\frac km&=1-p+K\sqrt{\frac{p(1-p)}{m}\log n}; \\
\frac{1-\frac km}{1-p}&=1+K\sqrt{\frac{p}{m(1-p)}\log n}.
\end{align*}
It is easy to check that $x>-0.1$ implies $\log (1+x)\geq x-2x^2/3$. On the event $B_m$ we have $100K^2\cdot\frac{1-p}{mp}\log n <1$, and hence $K\sqrt{\frac{1-p}{mp}\log n}<0.1$. Therefore
\begin{align*}
D\left(\frac km \bigg\| p\right)&\geq \bigg(p-K\sqrt{\frac{p(1-p)\log n}{m}}\bigg)\bigg(-K\sqrt{\frac{(1-p)\log n}{pm}}-\frac{2K^2(1-p)\log n}{3pm}\bigg)\\&
\quad+\bigg(1-p+K\sqrt{\frac{(1-p)p\log n}{m}}\bigg)\bigg(K\sqrt{\frac{p\log n}{(1-p)m}}-\frac{2K^2p\log n}{3(1-p)m}\bigg)\\
&=-K\sqrt{\frac{p(1-p)\log n}{m}}-\frac{2K^2(1-p)\log n}{3m}+\frac{K^2(1-p)\log n}{m}\\&\quad+\frac{2K^3}3\sqrt{\frac{(1-p)^3\log^3 n}{pm^3}}+K\sqrt{\frac{p(1-p)\log n}{m}}-\frac{2K^2p\log n}{3m}\\&\quad+\frac{K^2p\log n}{m}-\frac{2K^3}3\sqrt{\frac{p^3\log^3 n}{(1-p)m^3}}\\
&\geq\frac{K^2\log n}{3m}-\frac{2K^3p}{3}\cdot\sqrt{\frac{p\log^3 n}{(1-p)m^3}}.
\end{align*}
Similarly, we have
\[D\left(\frac{k'}{m} \bigg\| p\right)\geq  \frac{K^2\log n}{3m}-\frac{2K^3(1-p)}{3}\cdot \sqrt{\frac{(1-p)\log^3 n}{pm^3}}.\] 
Substituting this into the Chernoff bound, we obtain that for $q_1$ defined by equation \eqref{eq:epib} we have
\begin{align*}q_1&\leq \mathbb E\left(\exp\left(-\frac 13 K^2\log n+\frac{2K^3p}{3}\cdot\sqrt{\frac{p\log^3 n}{(1-p)m}}\right)\cdot \mathbb I_{B_m}\right)\\&+\mathbb E\left(\exp\left(-\frac 13 K^2\log n+\frac{2K^3(1-p)}{3}\cdot\sqrt{\frac{(1-p)\log^3 n}{pm}}\right)\cdot \mathbb I_{B_m}\right)\end{align*}
for $n$ large enough.
As for the first term: 
\[\frac{2K^3p}{3}\cdot\sqrt{\frac{p\log^3 n}{(1-p)m}}\leq \frac{K^3(\log n)^{3/2}}{\sqrt{ n(1-16\log n/n)}}\leq \frac{1}{12}K^2 \log n,\]
for $n$ large enough. Hence the first term is  $O(n^{-8})$, as we have chosen $K=6$.
In the exponent of the  second term,  since $pm>100K^2 \log  n$ holds on $B_m$, we get
\[\frac{2K^3(1-p)}{3}\cdot\sqrt{\frac{(1-p)\log^3 n}{pm}}\leq \frac{K^2}{15}\log n.\] 
Putting this together, we conclude that $q_1=O(n^{-8})$, which is a bound for the first term of \eqref{eq:py1}. 
The second term of \eqref{eq:py1} can be bounded as follows. 
\begin{align*}\mathbb P&\left(K\sqrt{mp(1-p)\log n}> K_1\sqrt{\frac{m}{n}}\log n\right)\leq \mathbb P\left(\sqrt p>\frac{K_1\sqrt{\log n}}{K\sqrt n}\right)\\
&=\mathbb P\left(p>\frac{K_1^2}{ K^2n}\log n\right)\leq n^{-8},\end{align*}
by equation \eqref{eq:kisp}, if $K_1^2\geq 16K^2=576$. This finishes the proof.
\end{proof}

Now we compare the differences of the proportions after $r$ steps and the further steps. This will give the order of the distance in the coupling. We define 
\[B=\bigg\{\frac{36000\log n}{n^{\alpha}}<p<\frac{16\log n}{n}\bigg\}\cap \{r>n^{\alpha}/10\}.\]
\begin{proposition} \label{prop:23r}Assuming $\alpha>1$, there exists $K_2, K_3, K_4, K_5>0$ such that for every fixed $1\leq i\leq n$ the following hold.
\begin{enumerate}[(a)]
\item
\[\mathbb P\left(\bigg\{|R_{i,t}-R_i^*|>K_2 \frac{\log n}{\sqrt{n^{\alpha+1}}}\bigg\}\cap B\cap \{t\geq r+n^{\alpha}\}\right)=O(n^{-8}).\]
\item \[\mathbb P\left(\bigg\{\sum_{t=r}^{\lfloor cn^2\rfloor}|R_{i,t}-R_i^*|>K_3 \log n \big( n^{3/2-\alpha/2}+ n^{\alpha-1}\big)\bigg\}\cap B\right)=O(n^{-6}).\]
\item \[\mathbb P\left(\bigg\{\sum_{t=r}^{\lfloor cn^2\rfloor}|R_{i,t}-R_i^*|>K_4 \log n\cdot \left( n^{3/2-\alpha/2}+n^{\alpha-1}\right)\bigg\}\cap \{r>n^{\alpha}/10\}\right)=O(n^{-6}).\]
\item We define 
\[\Delta_i=\sum_{t=r}^{\lfloor cn^2\rfloor}\left(|R_{i,t}-R_i^*|+R_{i,t}\sum_{k=1}^n  |R_{k,t+1}-R_k^*|\right).\]
Then for some $K_5>0$ we have 
\[\mathbb P\left(\bigg\{\Delta_i>K_5 \log^2 n\cdot \big( n^{3/2-\alpha/2}+n^{\alpha-1}\big)\bigg\}\cap \{r>n^{\alpha}/10\}\right)=O(n^{-5}).\]
\item For $K_5>0$ defined in $(d)$, we have 
\[\mathbb P\left(\Delta_i>K_5 \log^2 n\cdot \big( n^{3/2-\alpha/2}+n^{\alpha-1}\big)\right)=O(n^{-5}).\]

\end{enumerate}
\end{proposition}

\begin{proof}  We will  assume that $r<cn^2$; otherwise the sums become empty, and $\Delta_i=0$.

$(a)$  We will use the representation based on de Finetti's theorem together with the following decomposition.  
\begin{align*}|R_{i,t}-R_i^*|&=\bigg|\frac{X_{i,t}}{t+n}-\frac{X_i^*}{r+n}\bigg|=\bigg|\frac{X_{i,t}-X_i^*}{t+n}-X_i^*\cdot\frac{t-r}{(t+n)\cdot (r+n)}\bigg|\\&\leq 
\frac{|X_{i,t}-X_i^*-\mathbb E(X_{i,t}-X_i^*|p)|}{t+n}+\frac{|X_i^*-\mathbb E(X_i^*|p)|(t-r)}{(t+n)\cdot (r+n)}\\&\quad +\bigg|\frac{\mathbb E(X_{i,t}-X_i^*|p)}{t+n}-\frac{\mathbb E(X_i^*|p)(t-r)}{(t+n)(r+n)}\bigg|.\end{align*}
According to the representation, we  know that $X_{i,t}-X_i^*$ is a binomial random variable with parameters $m=t-r$ and $p$, given $p$ and $r$. We will use Lemma \ref{lem:binom} for this conditional distribution. Notice that $B\cap \{t\geq r+n^{\alpha}\}\subseteq B_m$, and $m\geq n$ in this case. Therefore for $K_1$ defined in Lemma \ref{lem:binom} we have
\begin{align*}\mathbb P&\left(\bigg\{|X_{i,t}-X_i^*-\mathbb E(X_{i,t}-X_i^*|p)|>K_1\sqrt{\frac{(t-r)}{n}}\log n\bigg\}\cap B\cap\{t\geq r+n^{\alpha}\}\bigg\vert p,r\right)\\&=O(n^{-8}).\end{align*}
It follows that 
\begin{equation}\label{eq:s1}\mathbb P\left(\bigg\{\frac{|X_{i,t}-X_i^*-\mathbb E(X_{i,t}-X_i^*|p)|}{t+n}>K_1\frac{1}{\sqrt{tn}}\log n\bigg\}\cap B\cap\{t\geq r+n^{\alpha}\}\right)=O(n^{-8}).\end{equation}
Similarly, $X_i^*-1$ is a binomial random variable with parameters $m=r$ and $p$, given $p$ and $r$. Again, we have that $B\cap \{t\geq r+n^{\alpha}\}\subseteq B_m$. Thus Lemma \ref{lem:binom} can be applied. We get that there exists $K_1'>0$ such that 
\[\mathbb P\left(\bigg\{|X_i^*-\mathbb E(X_i^*|p)|>K_1'\sqrt{\frac{r}{n}}\log n\bigg\}\cap B\cap\{t\geq r+n^{\alpha}\}\bigg\vert p,r\right)=O(n^{-8}).\]
This implies
\[\mathbb P\left(\bigg\{\frac{|X_i^*-\mathbb E(X_i^*|p)|(t-r)}{(t+n)(r+n)}>K_1'\sqrt{\frac{r}{(r+n)^2n}}\log n\bigg\}\cap B\cap\{t\geq r+n^{\alpha}\}\bigg\vert p,r\right)=O(n^{-8}).\]
In addition, using that $r>n^{\alpha}/10$ holds on the event $B$, we can write
\begin{equation}\label{eq:s2}\mathbb P\left(\bigg\{\frac{|X_i^*-\mathbb E(X_i^*|p)|(t-r)}{(t+n)(r+n)}>K_1'\sqrt{\frac{10}{n^{\alpha+1}}}\log n\bigg\}\cap B\cap\{t\geq r+n^{\alpha}\}\right)=O(n^{-8}).\end{equation}
Now we reformulate  the third term.  
\begin{align*}S=\bigg|&\frac{\mathbb E(X_{i,t}-X_i^*|p)}{t+n}-\frac{\mathbb E(X_i^*|p)(t-r)}{(t+n)(r+n)}\bigg|=\bigg|\frac{(t-r)p}{t+n}-\frac{(1+r p)(t-r)}{(t+n)(r+n)}\bigg|\\&=\frac{t-r}{(t+n)(r+n)}\cdot |p(r+n)-(1+r p)|=\frac{t-r}{(t+n)(r+n)}|np-1|.\end{align*}
By equation \eqref{eq:kisp} we obtain
\begin{align*}\mathbb P\left(\bigg\{S>\frac{160\log n}{n^{\alpha}}\bigg\}\cap B\right)&\leq \mathbb P\left(\bigg\{\frac{|np-1|}{r+n}>\frac{160\log n}{n^{\alpha}}\bigg\}\cap B\right)\\&\leq \mathbb P(|np-1|>16\log n)=O(n^{-8}).\end{align*}
Putting this together with equations \eqref{eq:s1} and \eqref{eq:s2}, we obtain that there exists $K_2'>0$ such that 
\[\mathbb P\left(\bigg\{|R_{i,t}-R_i^*|>K_2' \left(\frac{\log n}{\sqrt{tn}}+\frac{\log n}{\sqrt{n^{\alpha+1}}}+\frac{\log n}{n^{\alpha}}\right)\bigg\}\cap B\cap \{t>r+n^{\alpha}\}\right)=O(n^{-8}).\]
Since $\alpha>1$ and $t>r+n^{\alpha}$, for $n$ large enough, the middle term is the largest one, and we conclude that for some $K_2>0$
\[\mathbb P\left(\bigg\{|R_{i,t}-R_i^*|>K_2 \frac{\log n}{\sqrt{n^{\alpha+1}}}\bigg\}\cap B\cap \{t>r+n^{\alpha}\}\right)=O(n^{-8}).\]
This finishes the proof of $(a)$.
 
 $(b)$ It follows from part $(a)$ that 
 \[\mathbb P\left(\sum_{t=\lceil r+n^{\alpha}\rceil}^{cn^2} |R_{i,t}-R_i^*|>cK_2\log n \cdot n^{3/2-\alpha/ 2}\bigg\}\cap B\right)=O(n^{-6}). \]
   On $B$, we have $r>n^{\alpha}/10>n$, as $\alpha>1$, for large enough $n$. By equation \eqref{eq:maxrt} we get that 
   \[\mathbb P\left(\bigg\{ \sum_{t=r}^{\lfloor r+n^{\alpha}\rfloor} |R_{i,t}-R_i^*|>n^{\alpha}\cdot \frac{72}{n}\log n\bigg\}\cap B\right)\leq 2cn^{-6}=O(n^{-6}).\]
   The two equations together imply the statement.

 $(c)$ Similarly to the proof of Lemma \ref{lem:maxrt}, for every $t\geq n^{\alpha}/10$ we have
 \begin{align*}\mathbb P&\left(\bigg\{R_{i,t}>\frac{64000}{n^{\alpha}}\log n\bigg\}\cap  \bigg\{p\leq \frac{36000\log n}{n^{\alpha}}\bigg\}\right)\\&\leq\mathbb P\left(X_{i,t}>\frac {64000(t+n)}{n^{\alpha}} \log n\bigg\vert p\leq \frac{36000}{n^{\alpha}}\log n\right)\\
&\leq \frac{\mathbb E\big((1+(e-1)p)^t|p\leq \frac{36000}{n^{\alpha}}\log n\big)}{\exp(\log n\cdot 64000 (t+n)/n^{\alpha})}\\&\leq \frac{\exp\big((e-1)t\cdot\frac {36000} {n^{\alpha}} \log n\big)}{\exp(\log n\cdot {64000}(t+n)/n^{\alpha})}\\
&\leq \exp\left(((e-1)\cdot 36000-64000)\frac t {n^{\alpha}} \log n\right)\leq n^{-8}.\end{align*}
   Therefore, writing
   \[
   \ms{L}:=\left\{\sum_{t=r}^{\lfloor cn^2\rfloor}|R_{i,t}-R_i^*|>128000 cn^{2-\alpha}\log n\right\}\cap  \left\{\frac{p}{10^3}\leq \frac{36\log n}{n^{\alpha}}\right\}\cap \{r>n^{\alpha}/10\},
   \]
   we have
   \begin{equation}\label{eq:kicsip}\mathbb P\left(\ms{L}\right)=O(n^{-6}),\end{equation}
   because on the event $\{r>n^{\alpha}/10\}$ we have $t>n^{\alpha}/10$ in all terms (and the inequality is valid for $R_i^*=R_{i, r}$ as well). 
   
   For $K_4$ large enough (which may depend only on $c$), the condition
   \[\bigg\{\sum_{t=r}^{\lfloor cn^2\rfloor}|R_{i,t}-R_i^*|>K_4 \max \Big(n^{2-\alpha}\log n, \big(\log n\cdot n^{3/2-\alpha/2}+\log n\cdot n^{\alpha-1}\big)\Big)\bigg\}\cap \{r>n^{\alpha}/10\}\]
   implies that either the event in part $(b)$, or the event in inequality \eqref{eq:kicsip}, or $\{p>16\log n/n\}$ holds, according to the value of $p$. Notice that for $\alpha>1$ we have $2-\alpha<3/2-\alpha/2$, hence for large enough $n$ we can get rid of the maximum. Thus, combining these inequalities with part $(a)$ of Lemma \ref{lem:maxrt}, we get the statement of $(c)$. 
   
   $(d)$ For the first term of $\Delta_i$, we know this statement with constant $K_4$ from part $(c)$. We may assume that $n$ is so large that $n^{\alpha}/10\geq n$ holds. Then we can apply  Lemma \ref{lem:maxrt} to get
    \[\mathbb P\left(\bigg\{\max_{r\leq t\leq \lfloor cn^2\rfloor} R_{i,t}>\frac{16\log n}{n}\bigg\}\cap\{r>n^{\alpha}/10\}\right)=O(n^{-8}).\] 
      On the other hand, if $\max_{r\leq t\leq \lfloor cn^2\rfloor} R_{i,t}\leq \frac{16\log n}{n}$ holds and the second term of $\Delta_i$ is greater than the bound in $(d)$, then 
      \[\sum_{t=r}^{\lfloor cn^2\rfloor} \sum_{k=1}^n  |R_{k,t+1}-R_k^*|>\frac{K_5}{16} \log n\cdot n\cdot \big( n^{3/2-\alpha/2}+n^{\alpha-1}\big)\]
      holds. By choosing $K_5=16K_4$, this implies that for some $1\leq k\leq n$ we have 
\[\sum_{t=r}^{\lfloor cn^2\rfloor}  |R_{k,t+1}-R_k^*|>K_4 \log n\cdot   \big( n^{3/2-\alpha/2}+n^{\alpha-1}\big).\]
Putting this together with part $(c)$, this finishes the proof of $(d)$ (notice that $K_4$ does not depend on $i$).
       
 $(e)$ To see that $(d)$ implies $(e)$, we only have to check that 
 \begin{equation}\label{eq:rn}\mathbb P(r\leq n^{\alpha}/10)=O(n^{-5}).\end{equation}                       
  Recall that the random variable $r'=r+n$ has negative binomial distribution with parameters $n$ and $p_{\alpha}=1-\mathrm{exp}(-n^{-\alpha+1})$.   For $n$ large enough, the inequality $\mathbb P(r\leq n^{\alpha}/10)\leq \mathbb P(r'\leq n^{\alpha}/5)$ holds and we also have
  \begin{equation}\label{eq:palpha}\frac{1}{2n^{\alpha-1}}\leq \frac{1}{n^{\alpha-1}}-\frac 23\cdot \frac{1}{n^{2\alpha-2}}\leq p_\alpha=1-e^{-n^{-\alpha+1}}\leq\frac{1}{n^{\alpha-1}}.\end{equation}                     
 Notice that $r'$ can be expressed as the independent sum of $n$ geometric random variables supported on $\mb{N}^+$ with mean $m=1/p_{\alpha}$. Thus, we compare $r'/n$ to $n^{\alpha-1}/5$, which is less than the mean of the geometric random variables. Hence we can apply Cram\'er's theorem for $b=n^{\alpha-1}/5$. We obtain that 
 \[\mathbb P(r'\leq n^{\alpha}/5)\leq \exp\big(-n(\vartheta b-\log M(\vartheta))\big),\]
where $M(\vartheta)$ is the moment generating function of this geometric random variables, and $\vartheta$ minimizes the expression in the exponent. That is, we have 
\[M(\vartheta)=\frac{p_{\alpha}e^{\vartheta}}{1-(1-p_{\alpha})\vartheta}; \qquad \vartheta=\frac{1}{1-p_{\alpha}}-\frac{1}{b-1}.\]
This yields 
\[\mathbb P(r'\leq n^{\alpha}/5)\leq\mathrm{exp}\left(-nb\left(\frac{1}{1-p_{\alpha}}-\frac{1}{b-1}\right)+n\log \frac{p_{\alpha}e^{\vartheta}(b-1)}{1-p_{\alpha}}\right).\]
 It follows from inequality \eqref{eq:palpha} that for $n$ large enough we have 
 \[\mathbb P(r'\leq n^{\alpha}/5)\leq\mathrm{exp}\big(-n^{\alpha}/5+2n+n\log (2e^2/10)\big).\]
Since we assumed that $\alpha>1$, this implies inequality \eqref{eq:rn}.
      \end{proof}

{\bf Proof of Proposition $\ref{prop:23}$.} If $r>cn^2$, then both models give the empty graph and the distance is $0$; we will ignore this case. For $t$ odd, let $\mathbb I_{i,t}$ be the indicator of the following event: either vertex $i$ gets different edges at step $(t, t+1)$ in the coupling of model 2 and model 3, or it gets an edge in exactly one of the models. For $t$ even, let $\mathbb I_{i,t}=0$. We will be interested in  $Z_i=\sum_{t=r+1}^{\lfloor cn^2\rfloor} \mathbb I_{i,t}$. In addition, we define  
\[\mathcal G=\sigma \big(r; R_{i,t}: 1\leq i \leq n, 1\leq t \leq cn^2\big). \] 
Whenever $\mathbb I_{i,t}$ takes value $1$, we either choose vertex $i$ in exactly one of the models at step $t$ or $t+1$, or we choose vertex $i$ in both models, but it gets different pairs in the two models. 
Thus, by the definition of the coupling, we have that  
 \[\mathbb E(\mathbb I_{i,t}|\mathcal G)\leq |R_{i,t}-R_i^*|+|R_{i,t+1}-R_i^*|+R_{i,t}\sum_{k=1}^n |R_{k,t+1}-R^*_k|+R_{i,t+1}\sum_{k=1}^n |R_{k,t}-R^*_k|.\] 
  A slight modification of Proposition \ref{prop:23r} implies that for some $K_6>0$ we have 
 \begin{equation}\label{eq:k6s}\mathbb P\left(\sum_{t=r}^{\lfloor cn^2\rfloor} \mathbb E(\mathbb I_{i,t}|\mathcal G)>K_6 \log^2 n\cdot \big(n^{3/2-\alpha/2}+n^{\alpha-1}\big)\right)=O(n^{-5}).\end{equation}
 To see this, note that the sum for the first two terms for odd $t$ gives the first term of $\Delta_i$ defined in part $(d)$ of Proposition \ref{prop:23r}. The third term here corresponds to the second term of $\Delta_i$ with even $t$s omitted. Finally, for the fourth term it is easy to see that the proof of Proposition \ref{prop:23r} is valid if $t+1$ is replaced by $t-1$.
 
Let $D$ be event in equation \eqref{eq:k6s}, and let $k_n=K_6 \log^2 n\cdot \big(n^{3/2-\alpha/2}+n^{\alpha-1}\big)$. By using that $D\in \mathcal G$ and given $\mathcal G$, the indicators $\mathbb I_{i,t}$ are conditionally independent by the definition of the coupling, we obtain 
\begin{align*}\mathbb P(\{Z_i> 2k_n\}\cap \overline D)&\leq \mathbb P(Z_i> 2k_n|\overline{D})\leq  \frac{\mathbb E(e^{Z_i}|\overline{D})}{\exp\big(2k_n\big)}=\frac{\mathbb E(\mathbb E(e^{Z_i}|\mathcal G)|\overline{D})}{\exp\big(2k_n\big)}\\
&\leq \frac{\mathbb E\Big(\prod_{t=r}^{\lfloor cn^2\rfloor} (1+(e-1)\mathbb E(\mathbb I_{i,t}|\mathcal G))\Big| \overline{D}\Big)}{\exp\big(2k_n\big)}\\
&\leq \frac{\mathbb E\Big(\exp\big((e-1)\sum_{t=r}^{\lfloor cn^2\rfloor}\mathbb E(\mathbb I_{i,t}|\mathcal G) \big)\Big|\overline D\Big)}{\exp\big( 2k_n\big)}\\&\leq \frac{\exp\big((e-1)\cdot k_n \big)}{\exp\big(2k_n\big)}\leq \exp\big((e-3)k_n)\big)=O(n^{-5}).\end{align*}
Putting this together with equation \eqref{eq:k6s}, we get that 
 \begin{equation*}\mathbb P\left(\sum_{t=r}^{\lfloor cn^2\rfloor} \mathbb I_{i,t}>2K_6 \log^2 n\cdot \big(n^{3/2-\alpha/2}+n^{\alpha-1}\big)\right)=O(n^{-5}).\end{equation*}
 This immediately implies that 
 \begin{equation*}\mathbb P\left(\max_{1\leq i \leq n} \left(\sum_{t=r}^{\lfloor cn^2\rfloor} \mathbb I_{i,t}\right)>2K_6 \log^2 n\cdot \big(n^{3/2-\alpha/2}+n^{\alpha-1}\big)\right)=O(n^{-4}).\end{equation*}
 The sum of the indicators is at most $cn^2$. We conclude that 
 \begin{equation*}\mathbb E\left(\max_{1\leq i \leq n} \left(\sum_{t=r}^{\lfloor cn^2\rfloor} \mathbb I_{i,t}\right)\right)\leq 2K_6 \log^2  n\cdot \big(n^{3/2-\alpha/2}+n^{\alpha-1}\big)+O(n^{-2}).\end{equation*}
 
Since the definition of model 2 and model 3 is the same during the first $r-1$ steps, and we included all possible differences into the indicators, $\sum_{t=r-1}^{\lfloor cn^2\rfloor} \mathbb I_{i,t}\leq Z_i+1$ is an upper bound for $\sum_{j=1}^n |U_{ij}-V_{ij}|$, where $U_{ij}$ is the number of edges between $i$ and $j$ in model $2$, and $V_{ij}$ is the corresponding quantity in model $3$ (at the end of the whole process). By using Lemma \ref{lem:sorosszeg} we get the statement of Proposition \ref{prop:23}.\hfill $\square$

\subsection{Models $3$ and $4$}

\subsection*{Proof of Proposition \ref{prop:34}} Let $U_{ij}$ be the number of edges between $i$ and $j$ in model $3$, and $V_{ij}$ be the number of edges between them in model $4$. By using the notations introduced for the coupling of the two models, we have 
$U_{ij}-V_{ij}=N_{\tau}^{(ij)}-N_{cn^2}^{(ij)}$.
If  $\tau\geq cn^2$, then all the differences are nonnegative, and all of them are negative if $\tau<cn^2$. Thus 
\begin{equation}\label{eq:nij}\sum_{j=1}^n |U_{ij}-V_{ij}|=\bigg|\sum_{j=1}^n N_{\tau}^{(ij)}-N_{cn^2}^{(ij)}\bigg| \qquad (1\leq i \leq n).\end{equation} 
We will use the fact that by cumulating the independent Poisson processes assigned to the pairs of vertices we get a Poisson process with rate $2\sum_{i<j} R_i^*R_j^*+\sum_i R_i^*=1$. In addition, the types $(ij)$ of  the events are independent of the moments when they occur. 
Let $N_s$ be the total number of events until time $s$; i.e.\ $N_s=\sum_{i\leq j} N_s^{(ij)}$, which has Poisson distribution with parameter $s$. Since there are $\lfloor cn^2\rfloor-r$ events in the cumulated process until $\tau$, there are $|\lfloor cn^2\rfloor-r-N_{cn^2}|$ events between $\tau$ and $cn^2$.   On the other hand, independently of each other, all these events increase $\big|\sum_{j=1}^n N_{\tau}^{(ij)}-N_{cn^2}^{(ij)}\big|$ by $1$ with probability $p_i'=R_i^*(R_i^*+2\sum_{j\neq i} R_j^*)\leq 2R_i^*$. We conclude that the quantity in equation \eqref{eq:nij} has binomial distribution with parameters $|\lfloor cn^2\rfloor-r-N_{cn^2}|$ and $p_i'\leq 2R_i^*$ conditionally with respect to $N_{cn^2}$ and $(R_j^*)_{j=1}^n$. Let $F_i$ be the following event: 
\[F_i=\bigg\lbrace R_i^*\leq\frac{36}{n}\log n\bigg\rbrace\cap \big\lbrace r<2n^{\alpha}\rbrace \cap \big\lbrace |N_{cn^2}-cn^2|<n^{\alpha}\big\rbrace.\]
By using the moment generating function of the binomial distribution, we obtain 
\begin{align*}\mathbb P&\left(\bigg\lbrace \sum_{j=1}^n |U_{ij}-V_{ij}|>128 \log n\cdot n^{\alpha-1}\bigg\rbrace\cap F_i\right)\leq \mathbb E\left(\frac{(1+(e-1)p_i')^{|\lfloor cn^2\rfloor-r-N_{cn^2}|}}{\exp(128 \log n\cdot n^{\alpha-1})}\cdot \mathbb I_{F_i}\right)\\
&\leq \mathbb E(\exp((e-1)(72\log n/n)\cdot 3n^{\alpha}-128 \log n\cdot n^{\alpha-1}\big)=O(n^{-6}).\end{align*}
It follows from  part $(b)$ of Lemma \ref{lem:maxrt} and equation \eqref{eq:rn} that $\mathbb P(R_i^*>36\log n/n)=O(n^{-6})$. Similarly to the proof of equation \eqref{eq:rn} in part $(e)$ of Proposition \ref{prop:23r}, it can be shown that $\mathbb P(r\geq 2n^{\alpha})=O(n^{-5})$; one can use Cram\'er's large deviation theorem and the fact that the expectation of $r$ is smaller than $n^{\alpha}$. Finally, recall that $N_{cn^2}$ has Poisson distribution with parameter $cn^2$. We can think of it as the independent sum of $n^2$ Poisson random variables with parameter $c$, and apply Cram\'er's theorem. That is, 
\begin{align*}\mathbb P(N_{cn^2}-cn^2>n\log n)&=\mathbb P(N_{cn^2}/n^2>c+\log n/n)\\&\leq \exp\big(-n^2(\vartheta (c+\log n/n)-\log M(\vartheta)),\big)\end{align*}
where $M$ is the moment generating function of $\mathrm{Poisson}(c)$, and we can choose $\vartheta$ to minimize the expression on the right hand side. By using $\log M(\vartheta)=c(e^{\vartheta}-1)$ and $\vartheta=\log(1+\log n/n)$, it follows that this probability is also $O(n^{-6})$. The same argument works for 
$\mathbb P(N_{cn^2}-cn^2<-n\log n)$. On the other hand, $\alpha>1$, hence $n^{\alpha}>n\log n$ for large $n$. 

Putting this together, we obtain that $\mathbb P(\overline{F_i})=O(n^{-6})$, and 
\[\mathbb P\left(\sum_{j=1}^n |U_{ij}-V_{ij}|>128 \log n\cdot n^{\alpha-1}\right)=O(n^{-6}).\] 
Since the total sum cannot be larger than $cn^2$, we get Proposition \ref{prop:34} similarly to the arguments in the previous section. \hfill $\square$

\subsection{Models $4$ and $5$}
\subsection*{Proof of Proposition \ref{prop:45}}
The expected value $\mathbb E\big(d_{\boxtimes}\big(\mathbb G_4(n, \alpha), \mathbb G_5(n)\big)\big)$ can be split according to the value of $r$ as follows.
\begin{align*}
\mathbb E\big(d_{\boxtimes}\big(\mathbb G_4(n, \alpha), \mathbb G_5(n)\big)\big)&=
\mathbb E\big(\left.d_{\boxtimes}\big(\mathbb G_4(n, \alpha), \mathbb G_5(n)\big)\right|r>cn^2\big)\mb{P}(r>cn^2)\\&+
\mathbb E\big(\left.d_{\boxtimes}\big(\mathbb G_4(n, \alpha), \mathbb G_5(n)\big)\right|r\leq cn^2\big)\mb{P}(r\leq cn^2).
\end{align*}
The second term is zero by the coupling, whilst the first is
\[
\mathbb E\big(\left.d_\boxtimes(0,\mathbb G_5(n))\right|r>cn^2\big)\mb{P}(r>cn^2).
\] 
To bound this, note that we always have
\begin{align*}
d_\boxtimes(0,\mathbb G_5(n))\leq \frac 1n\bigg( 2\sum_{1\leq i<j\leq n}Z_{ij}+\sum_{1\leq i\leq n}Z_{ii}\bigg).
\end{align*}
But we have by the definition of the variables $Z_{ij}$
\begin{align*}
\mb{E}\left( \left.2\sum_{1\leq i<j\leq n}Z_{ij}+\sum_{1\leq i\leq n}Z_{ii}\right|R_1^*,R_2^*,\ldots,R_n^*\right)&=
2\sum_{1\leq i<j\leq n}cn^2R_i^*R_j^*+\sum_{1\leq i\leq n}cn^2(R_i^*)^2/2\\
&\leq cn^2\left(\sum_{1\leq i\leq n} R_i^*\right)^2=cn^2,
\end{align*}
whence
\begin{align*}
\mathbb E\big(d_{\boxtimes}\big(\mathbb G_4(n, \alpha), \mathbb G_5(n)\big)\big)&\leq
cn\mb{P}(r>cn^2)\leq cn\mb{P}(r'>cn^2).
\end{align*}
Since $r'$ is, as noted before, the sum of $n$ independent geometric distributions of parameter $p_\alpha=1-e^{-\frac{1}{n^{\alpha-1}}}$ supported on $\mb{N}^+$, we have
\begin{equation*}
\mb{P}(r'>cn^2)\leq \mb{P}(\mr{Geom}(p_\alpha)>cn)=(1-p_\alpha)^{\lceil cn\rceil}\leq e^{-\frac{cn}{n^{\alpha-1}}}=e^{-cn^{2-\alpha}}.
\end{equation*}

Provided  $\alpha<2$, this yields $cn^2e^{-cn^{2-\alpha}}\leq O(n^{-10})$.

\subsection{Models $5$ and $6$}

To be able to bound the jumble distance, we have to deal with each of the random variables $H_{ij}^*$. Recall that $\mc{F}$ denoted the $\sigma$-algebra generated by the $\xi_i$ and $R_i^*$, $1\leq i\leq n$. By our coupling we may write for each $1\leq i\leq j\leq n$
\begin{equation*}
\mb{E}\left(H_{ij}^*\right)=\mb{E}\left(\mb{E}\left(\left.H_{ij}^*\right|\xi_i,\xi_j\right)\right)
=\frac{2-\delta_{ij}}{2}\mb{E}\left(
cn^2\left|
\frac{C_iC_j}{\left(\sum_{k=1}^n C_k\right)^2}-\frac{\xi_i\xi_j}{n^2}
\right|\right).
\end{equation*}

\begin{lemma}\label{le:fact_mom} Provided $\alpha\geq 1/2$, we have for all non-negative integers $b\in\mb{N}_0$
\begin{equation}\label{eqn:fact_mom_exp}
\mb{E}\left(
(H_{ij}^*)^{(b)}
\right)
\leq K_b\bigg(\frac{1}{n^{b/2}}+\frac{1}{n^{(\alpha-1)b}}\bigg),
\end{equation}
where $k^{(b)}$ denotes the $b^{th}$ factorial moment for any $k\in\mb{N}_0$, i.e.\ $k^{(b)}=k(k-1)\ldots(k-b+1)$.
\end{lemma}
\begin{proof}
It is known that for any $b\in\mb{N}^+$ we have $\mb{E}\left(\mr{Pois}(\lambda)^{(b)}\right)=\lambda^b$. Suppose now that $b\geq 1$. By the law of total expectation, we have
\begin{align*}
\nonumber\mb{E}\left(
(H_{ij}^*)^{(b)}
\right)&=\mb{E}\left(\mb{E}\left(
\left.
(H_{ij}^*)^{(b)}
\right|\mc{F}
\right)\right)=\label{eqn:fff}
\frac{(2-\delta_{ij})^b}{2^b}\mb{E}\left(
c^bn^{2b}\left|
\frac{C_iC_j}{\left(\sum_{k=1}^n C_k\right)^2}-\frac{\xi_i\xi_j}{n^2}
\right|^b
\right)\\
&\leq2^bc^bn^{2b}(F_1+F_2),
\end{align*}
where
\[
F_1:=\mb{E}\left(\left|\frac{\xi_i\xi_j}{(\sum_k \xi_k)^2}-\frac{\xi_i\xi_j}{n^2}\right|^b\right); \qquad
F_2:=\mb{E}\left(\left|\frac{C_iC_j}{(\sum_k C_k)^2}-\frac{\xi_i\xi_j}{\sum_k \xi_k}\right|^b\right), 
\]
and we made use of the power mean inequality in the form $(a_1+a_2)^b\leq 2^{b-1}(a_1^b+a_2^b)$.
Note that we may consider $F_1$ as the error that stems from the randomization in the denominator, whilst $F_2$ captures the error that comes from the rounding $\xi_in^{\alpha-1}\to C_i$.

Let us first bound $F_1$. It is known that for the i.i.d.\ exponential variables $\xi_i$, their sum $\sum \xi_k$ is independent from the ratios $\xi_i/\sum \xi_k$. Hence
\begin{align*}
F_1&=\mb{E}\left(\left|\frac{\xi_i\xi_j}{(\sum_k \xi_k)^2}-\frac{\xi_i\xi_j}{n^2}\right|^b\right)=\frac{1}{n^{2b}}\mb{E}\left(\frac{\xi_i^b\xi_j^b}{(\sum_k \xi_k)^{2b}}\left|n^2-\left(\sum_k \xi_k\right)^2\right|^b\right)\\
&=\frac{1}{n^{2b}}\mb{E}\left(\frac{\xi_i^b\xi_j^b}{(\sum_k \xi_k)^{2b}}\right)\mb{E}\left(\left|n^2-\left(\sum_k \xi_k\right)^2\right|^b\right).
\end{align*}
Also, we have $\frac{\xi_i}{\sum_k \xi_k}\sim\mr{Beta}(1,n-1)$. The first term can thus be bounded by
\begin{equation*}\label{eq:f11}
F_{1,1}:=\mb{E}\left(\frac{\xi_i^b\xi_j^b}{(\sum_k \xi_k)^{2b}}\right)\leq \mb{E}\left(\frac{\xi_i^{2b}}{(\sum_k \xi_k)^{2b}}\right)=\frac{(2b)!}{(n+1)(n+2)\ldots(n+2b)}.
\end{equation*}
We have that given $n$ i.i.d.\ random variables with expectation $0$, and an integer $\nu\geq 2$, the $\nu^{th}$ moment of their sum is bounded by $K^2n^{\nu/2}$, with $K$ depending only on the distribution (see e.g.\ \cite{bahr, phall}). In addition, $\sum_k \xi_k\sim\mr{Gamma}(n,1)$. The second term can therefore be bounded by
\[\begin{split}F_{1,2}&=\mathbb E\left(\left|n^2-\left(\sum_k \xi_k\right)^{\!2}\right|^b\,\right)=\mathbb E\left(\left|\sum_k \xi_k-n\right|^b\left(\sum_k \xi_k+n\right)^{\!b}\,\right)\\
&\leq \sqrt{\mathbb E\left(\left|\sum_k \xi_k-n\right|^{2b}\right)} \sqrt{\mathbb E\left(\left(\sum_k \xi_k+n\right)^{\!2b}\right)}\\&\leq K n^{b/2} \cdot 2^{\frac{2b-1}{2}}\sqrt{\mathbb E\left(\sum_k \xi_k\right)^{2b}+n^{2b}}\\
&\leq K2^b n^{b/2}\sqrt{\frac{(2b+n-1)!}{(n-1)!}+n^{2b}}\leq K 2^b n^{b/2} \sqrt{(n+2b)^{2b}+n^{2b}}\\&\leq K 4^b n^{b/2} \sqrt{2n^{2b}+(2b)^{2b}}.\end{split}\]

Thus we obtain
\[
(2c)^bn^{2b}F_1\leq K(8c)^b\cdot \frac{(2b!)n^{b/2}\cdot\sqrt{2n^{2b}+(2b)^{2b}}}{(n+1)(n+2)\ldots(n+2b)}.
\]
For a fixed $b$, this means
\begin{equation}\label{eq:f1}(2c)^bn^{2b}F_1\leq \frac{K_b'}{n^{b/2}}.\end{equation}
Let us now turn to the term $F_2=\mathbb E\left(\left|\frac{C_iC_j}{(\sum_k C_k)^2}-\frac{\xi_i\xi_j}{\sum_k \xi_k}\right|^b\right)$. 
The first idea is to get rid of the absolute value by observing that if we have random variables $v_1,v_2,v_3$ such that $v_1\geq v_2\geq v_3$ and $v_1\geq 0\geq v_3$, then for any $b\in\mb{N}^+$ we have
\[
\mb{E}(|v_1|^b)+\mb{E}(|v_3|^b)\geq \mb{E}(|v_2|^b).
\]
The role of $v_2$ shall be played by $\frac{C_iC_j}{(\sum_k C_k)^2}-\frac{\xi_i\xi_j}{\sum_k \xi_k}$.

Using the fact that by the rounding, $n^{\alpha-1}\xi_k\leq C_k\leq n^{\alpha-1}\xi_k+1$ for each $1\leq k\leq n$, we have
\[\begin{split}\frac{C_iC_j}{(\sum_k C_k)^2}-\frac{\xi_i\xi_j}{(\sum_k \xi_k)^2}&\leq \frac{(n^{\alpha-1}\xi_i+1)(n^{\alpha-1}\xi_j+1)}{n^{2\alpha-2}(\sum_k \xi_k)^2}-\frac{\xi_i\xi_j}{(\sum_k \xi_k)^2}\\ 
&\leq \frac{n^{\alpha-1}(\xi_i+\xi_j)+1}{n^{2\alpha-2}(\sum_k \xi_k)^2},
\end{split}\]
and so we can have
\[
\frac{n^{\alpha-1}(\xi_i+\xi_j)+1}{n^{2\alpha-2}(\sum_k \xi_k)^2}
\]
play the role of $v_1$.
Applying first the power mean inequality, and using that the reciprocal of the sum $\sum \xi_k$ has inverse gamma distribution, whilst the ratio is a $\mr{Beta}(1,n-1)$ distribution independent of it, for $n$ large enough, we obtain 
\begin{align*}
\mb{E}(|v_1|^b)&\leq 2^{b-1}\frac{1}{n^{(\alpha-1) b}}\mathbb E\left(\bigg(\frac{\xi_i+\xi_j}{(\sum_k \xi_k)^2}\bigg)^b\right)+2^{b-1}\frac{1}{n^{2(\alpha-1) b}}\mathbb E\left(\frac{1}{(\sum_k \xi_k)^{2b}}\right)\\ 
&\leq \frac{4^{b}}{n^{(\alpha-1) b}}\mathbb E\left(\bigg(\frac{\xi_i}{(\sum_k \xi_k)^2}\bigg)^b\right)+\frac{2^{b-1}}{n^{2(\alpha-1) b}}\mathbb E\left(\frac{1}{(\sum_k \xi_k)^{2b}}\right)\\
&\leq \frac{4^{b}}{n^{(\alpha-1) b}}
\mathbb E\left(\bigg(\frac{\xi_i}{\sum_k \xi_k}\bigg)^{b}\right)
\mathbb E\left(\bigg(\frac{1}{\sum_k \xi_k}\bigg)^{b}\right)
+\frac{2^{b-1}}{n^{2(\alpha-1) b}}\mathbb E\left(\frac{1}{(\sum_k \xi_k)^{2b}}\right)\\
&\leq \frac{4^{b}}{n^{(\alpha-1) b}}\cdot\frac{b!}{(n+1)(n+2)\ldots(n+b)}\cdot\frac{1}{(n-1)\ldots(n-b)}\\
&\quad+\frac{2^{b-1}}{n^{2(\alpha-1) b}}\frac{1}{(n-1)\ldots(n-2b)}\leq  \frac{C_b}{n^{(\alpha+1)b}}\left(1+\frac{1}{n^{(\alpha-1) b}}\right).
\end{align*}

Again by the rounding, we have the lower bound 
\[\begin{split}\frac{C_iC_j}{(\sum_k C_k)^2}&-\frac{\xi_i\xi_j}{(\sum_k \xi_k)^2}\geq \frac{n^{2(\alpha-1)}\xi_i\xi_j}{(n^{(\alpha-1)}\sum_k \xi_k+n)^2}-\frac{\xi_i\xi_j}{(\sum_k \xi_k)^2}\\ 
&= -\xi_i\xi_j\frac{(n^{(\alpha-1)}\sum_k \xi_k+n)^2-n^{2(\alpha-1)}(\sum_k \xi_k)^2}{(n^{(\alpha-1)}\sum_k \xi_k+n)^2(\sum_k \xi_k)^2}\\
&=-\xi_i\xi_j\frac{2(n^{(\alpha-1)}\sum_k \xi_k+n)n}{(n^{(\alpha-1)}\sum_k \xi_k+n)^2(\sum_k \xi_k)^2}.\end{split}\]

Here it is clear that the last expression is negative, so let's continue without the minus sign.
\begin{align*}
\xi_i\xi_j\frac{(2n^{(\alpha-1)}\sum_k \xi_k+n)n}{(n^{(\alpha-1)}\sum_k \xi_k+n)^2(\sum_k \xi_k)}&\leq
\xi_i\xi_j\frac{2n}{(n^{(\alpha-1)}\sum_k \xi_k+n)(\sum_k \xi_k)^2}\leq \frac{2n\xi_i\xi_j}{n^{(\alpha-1)}(\sum_k \xi_k)^3}.
\end{align*}

So the role of $-v_3$ will be played by
\[
\frac{2\xi_i\xi_j}{n^{(\alpha-2)}(\sum_k \xi_k)^3}.
\]
We use that the sum is independent of the proportions, use inequality \eqref{eq:f11}, the Cauchy--Schwarz inequality and the moments of the Gamma distribution:
\begin{align*}
\mb{E}(|v_3|^b)&\leq  \frac{2^b}{n^{(\alpha-2)b}}\mathbb E\bigg(\Big(\frac{\xi_i\xi_j}{(\sum_k \xi_k)^2}\Big)^b\bigg)\mathbb E\bigg(\frac{1}{(\sum_k \xi_k)^b}\bigg)\\&\leq \frac{2^b}{n^{(\alpha-2)b}}\cdot\frac{(2b)!}{(n+1)\ldots(n+2b)}\cdot\frac{1}{(n-1)\ldots(n-b)}\leq \frac{C_b'}{n^{(\alpha+1)b}}.
\end{align*}

Hence $(2c)^bn^{2b}F_2\leq \frac{K''_b} {n^{(\alpha-1)b} }$, and summing up we obtain
\[
\mb{E}((H_{ij}^*)^{(b)})\leq \frac{K'_b}{n^{b/2}}+\frac{K''_b}{n^{(\alpha-1) b}}\leq {K_b}\bigg(\frac{1}{n^{b/2}}+\frac{1}{n^{(\alpha-1)b}}\bigg).
\qedhere\]
\end{proof}

\subsection*{Proof of Proposition \ref{prop:56}}

Recall that in the coupling of model $5$ and $6$, the absolute value of the difference of the number of edges between $i$ and $j$ is $H_{ij}^*$. By Lemma \ref{le:fact_mom} with $b=1$,  for some $K_1>0$, for every fixed $i$ we have 
\[
\mb{E}\left(\sum_{j=1}^n  H^*_{ij} \right)=n\mb{E}(H_{ij}^*)\leq K_1 \big(n^{1/2}+n^{2-\alpha}\big).
\]
Let now $\varrho_{ij}:=H_{ij}^*\wedge 3$, and $\sigma_i:=\sum_{j=1}^n \varrho_{ij}$. Clearly we have
\[
m:=\mb{E}(\sigma_i)\leq \mb{E}\left( \sum_{j=1}^n  H^*_{ij} \right)\leq K_1\big(n^{1/2}+n^{2-\alpha}\big).
\]
For fixed  $i$, conditionally on $\mc{F}=\sigma\{ \xi_j, R_j^*; 1\leq j \leq n\}$, the random variables $\varrho_{ij}$ ($1\leq j\leq n$) are independent. Since they fall between $0$ and $3$, by the Hoeffding inequality we have
\[
\mb{P}(|\sigma_i-m|\geq s|\mc{F})\leq 2\exp \left(-\frac{2s^2}{9n}\right)
\]
for any $s\geq 0$.
Using the same constant $K_1$ as above, and choosing $s:=9\sqrt{n\log n}$, we have by the bound on $m$ that
\begin{align*}\mathbb P(\sigma_i\geq (9+K_1)\sqrt {n \log n}+K_1n^{2-\alpha})&\leq \mathbb P(|\sigma_i-m|\geq 9 \sqrt {n\log n})\\&\leq 
2\exp\bigg(-\frac{2\cdot 81 n\log n}{9 n}\bigg)=2n^{-18}=O(n^{-4}).\end{align*}

A trivial  bound then yields 
\[\mathbb P(\max_{1\leq i \leq n}\sigma_i\geq (9+K_1)\sqrt {n \log n}+K_1n^{2-\alpha})=O(n^{-3}).\]
Since $\sigma_i\leq 3n$ always holds, we obtain
\[\mathbb E(\max_{1\leq i\leq n} \sigma_i)\leq (9+K_1) {\sqrt {n\log n}}+K_1n^{2-\alpha}+O(1)\leq K' {\sqrt {\log n}}\big(n^{1/2}+n^{2-\alpha}\big).\]

It is clear that $H_{ij}^*\leq \varrho_{ij}+(H_{ij}^*)^{(3)}$, since whenever $H^*_{ij}>3$, its $3$rd factorial moment is positive, and strictly larger than $H_{ij}^*$ itself.  
Therefore
\[\sum_{j=1}^n  H_{ij}^* \leq \sigma_i+\sum_{j=1}^n (H_{ij}^*)^{(3)}\quad \Rightarrow \quad\left(\max_{1\leq i\leq n} \sum_{j=1}^n H_{ij}^* \right)\leq \max_{1\leq i\leq n} \sigma_i+\sum_{i,j}(H_{ij}^*)^{(3)}.\]

From the above, together with inequality (\ref{eqn:fact_mom_exp}) :
\begin{align*}\mathbb E\left(\max_{1\leq i\leq n} \sum_{j=1}^n H_{ij}^* \right)&\leq  \mathbb E(\max_{1\leq i\leq n}\sigma_i)+n^2 \cdot \mathbb E\big((H_{ij}^*)^{(3)}\big)\leq \frac{K''}{2}\sqrt{\log n}\big(n^{1/2}+n^{2-\alpha}+n^{5-3\alpha}\big) \\
&\leq K''\sqrt{\log n} \big(n^{1/2}+n^{5-3\alpha}\big),
\end{align*}
where the last inequality follows from a weighted AM-GM.

Finally, Lemma \ref{lem:sorosszeg} concludes the proof. \hfill $\square$

\subsection{Models $6$ and $7$}

We have that $\mb{G}_6$ and $\mb{G}_7$ coincide everywhere but the main diagonal, and it is then easy to see that
\[
d_{\boxtimes}\big(\mathbb G_6(n), \mathbb G_7(n)\big)=\frac 1n\cdot \max_{1\leq i\leq n} Y_{ii}.
\]

\subsection*{Proof of Propositon $\ref{prop:67}$} 
Recall that $Y_{ii}$ has Poisson distribution with parameter $c\xi_i^2$, where $\xi_i$ has $\exp(1)$ distribution. 
Assume first that $\zeta>0$ is fixed, and $X\sim \mr{Pois}(\zeta)$. Then
\begin{equation*}
\sum_{k=y+1}^{\infty}\mathbb P(X\geq k)\leq \sum_{k=y+1}^{\infty} k \mathbb P(X= k)=\sum_{k=y+1}^{\infty} \frac{\zeta^k}{(k-1)!}e^{-\zeta} \leq\zeta \sum_{k=z}^{\infty} \frac{\zeta^{k}}{k!}e^{-\zeta}=\zeta \mathbb P(X\geq y).
\end{equation*}
We will use the factorial moments of the Poisson distribution again. For every fixed $i$ and integers $y>b>0$ for some $K(b)>0$ we have 
\begin{align*}
\sum_{k=y+1}^{\infty}\mathbb P(Y_{ii}\geq k)&=\mathbb E\bigg(\mathbb E\bigg( \sum_{k=y+1}^{\infty}\mathbb P(Y\geq k)\bigg\vert \xi_i\bigg) \bigg) \leq \mathbb E(c\xi_i^2 \mathbb P(Y^{(b)}\geq y^{(b)}|\xi_i)) \\
&\leq \mathbb E\bigg(c\xi_i^2\frac{\xi_i^{(b)}}{y^{(b)}} \bigg)=\frac{K(b)}{y^{(b)}},
\end{align*}
because the exponential distribution has finite moments.

For an arbitrary function $f:\mb{N}^+\to\mb{N}^+$ we may apply the above inequality to obtain
\begin{align*}
\mathbb E\big(\max_{1\leq i \leq n} |Y_{ii}|\big)&=\sum_{k=1}^{\infty} \mathbb P\big(\max_{1\leq i \leq n} |Y_{ii}|\geq k\big)\leq f(n)+\sum_{k=f(n)+1}^{\infty }n \mathbb P(Y_{11}\geq k)\leq f(n)+n\frac{K(b)}{f(n)^{(b)}}.
\end{align*}
Let now $N\in\mb{N}^+$ be fixed, set $f(n):=n^{1/5}$ and $b:=4$. For $n$ large enough (such that $n-3\geq n/2$) and $f(n)^{(4)}\geq f(n)^4/16$, this yields
\[\mathbb E\big(\max_{1\leq i \leq n} |Y_{ii}|\big)\leq n^{1/5}+16K(4)n^{1-4/5}\leq n^{1/4}(K(4)+1)).\]
Lemma \ref{lem:sorosszeg} concludes the proof. \hfill $\square$

\section{Discussion} 

\label{discussion}

Our main theorem shows that the classical dense preferential attachment graph model yields random graphs that are close to the random graph model obtained through the PAG-graphon, the limit object in the multigraph homomorphism sense of the random sequence $\mb{G}_{\mr{PAG}}$. They are not indistinguishable though (we provide a lower bound on their distance below), and they each have their own advantages for applications.

The random graphs $\mb{G}_{\mr{PAG}}$ have the advantage that the number of edges is deterministic, but contrarily to the sparse PAG models, one cannot easily generate a growing family of graphs $\mb{G}_{\mr{PAG}}(n)$. For the graphon induced $\mb{G}_W^\circ$, the number of edges is random, though still asymptotically concentrated around the expected value. Also, the way it is generated does not carry the preferential attachment flavour. This may be an advantage from the simulation point of view: the random variables in the model can be generated simultaneously, without the $cn^2$ steps that have to be performed after each other in the PAG model.\\
However, it is possible to couple the elements of the sequence $\mb{G}_W^\circ(n)$ (or $\mb{G}_W(n)$) so that we obtain a growing sequence  (and still keep the convergence with probability 1). Indeed, passing from $n$ to $n+1$ only means that we have to generate the random variable $\xi_{n+1}$, independently of the previous $\xi_i$-s, and then generate the appropriate Poisson random variables $Y_{j(n+1)}$ for $1\leq j\leq n+1$. This coupling shows that adding an extra vertex and extending $\mb{G}_W^\circ(n)$ to $\mb{G}_W^\circ(n+1)$ can be performed easily. It seems that this does not hold for the  $\mb{G}_{\rm PAG}$ model.

Unfortunately, we do not have a lower bound for the jumble norm distance of $\mb{G}_\mr{PAG}(n)$ and $\mb{G}_W(n)$ 
 that matches the upper bound given in Theorem \ref{thm:main}. Recall that we there obtained  $O(n^{-1/3}\log ^2 n)$ as an upper bound for a particular coupling. On the other hand, there is a universal lower bound of $O(n^{-1})$, which holds for every coupling, and also for both for the random graphs $\mb{G}_W(n)$ and $\mb{G}_W^\circ(n)$. The exponents are quite far from each other, but the arguments used for the lower bound use very little of the structure of the graphs.
We present a short argument giving this lower bound for both $\mb{G}_W^\circ(n)$ and $\mb{G}_W(n)$.

If we take $S=T=\{1, 2, \ldots, n\}$ in Definition \ref{def:jumble}, then we obtain a lower bound for the jumble norm distance of $\mathbb G_{\rm PAG}$ and $\mathbb G_{\rm W}$ by understanding the difference of the number of edges. The main point is that the distribution of this quantity does not depend on the coupling.  In $\mathbb G_{\rm PAG}(n)$, the number 
of edges is deterministic and it is equal to $\lfloor \lfloor cn^2\rfloor/2\rfloor$.
We denote by $\mathcal E$ the number of edges in the $\mathbb G_{\rm W}(n)$ graph model.   Let $\mathcal G$ be the $\sigma$-algebra generated by $\xi_1, \ldots, \xi_n$ (recall that the latter random variables are independent and have exponential distribution with parameter $1$). Then, conditionally with respect to $\mathcal G$, the random variable $\mathcal E$ has Poisson distribution with parameter $c\sum_{1\leq i<j\leq n} \xi_i \xi_j$. Hence $\mathbb E(\mathcal E)=cn(n-1)/2$ by the law of total expectation.

In any coupling of these two models, by $S=T=\{1, 2, \ldots, n\}$ we have 
\[d_{\boxtimes}\big(\mathbb G_{\rm PAG}(n), \mathbb G_{\rm W}(n)\big)\geq \frac{1}{n^2} \mathbb E\big(\big|\lfloor \lfloor cn^2\rfloor/2\rfloor-\mathcal E\big|\big).\]
Notice that 
\[\mathbb E\big(\big|\lfloor \lfloor cn^2\rfloor/2\rfloor-\mathcal E\big|\big)\geq \big|\mathbb E\big(\lfloor \lfloor cn^2\rfloor/2\rfloor-\mathcal E\big|\big)\big)\big|=\big|\lfloor \lfloor cn^2\rfloor/2\rfloor-cn(n-1)/2\big|\geq c_0 n\]
for an appropriate positive number $c_0$. This holds for every coupling; therefore the exponent in Theorem \ref{thm:main} cannot be smaller than $-1$.

The previous argument relies on the fact the expected number of edges is different in the two models, due to the lack of loops in the $\mathbb G_W$ model. For the $\mb{G}_\mr{PAG}$ and the $\mathbb G_W^{\circ}$ models, although the expected number of edges are equal to each other, one can prove that the jumble norm distance is still at least $\frac{1}{e^2}\sqrt{\frac c2}\cdot\frac 1n$ for every coupling. The key point is to use the formula for the central absolute moment of the Poisson distribution and see that it is at least constant times the square root of the parameter.

To see this, we have to consider the random variable $\mathcal E^{\circ}$, which is the number of edges in $\mathbb G_W^{\circ}$. It has Poisson  distribution with parameter $c\sum_{1\leq i<j\leq n}\xi_i \xi_j+\frac c2\sum_{i=1}^n \xi_i^2$ conditionally with respect to $\mathcal G$ (recall Definition \ref{def:gwn}). For sake of simplicity, let $\eta$ be a Poisson($\lambda$) distributed random variable, and $m>0$. First notice that 
\[\mathbb E(|\eta-m|)\geq |\mathbb E(\eta-m)|=|\lambda-m|.\]
On the other hand, by using the formula for the central absolute moment of the Poisson distribution and the well-known upper bound version of Stirling's formula, we have 
\[\mathbb E(|\eta-\lambda|)=2e^{-\lambda}\frac{\lambda^{\lfloor \lambda\rfloor+1}}{\lfloor\lambda\rfloor!}\geq 2e^{-\lambda}\frac{\lambda^{\lfloor \lambda\rfloor+1}}{e\cdot\lfloor \lambda\rfloor^{\lfloor \lambda\rfloor+1/2}\cdot e^{-\lfloor \lambda\rfloor}}\geq 2e^{-\lambda+\lfloor \lambda\rfloor-1}\sqrt{\lambda}\geq \frac{2}{e^2}\sqrt{\lambda}.\]
 Putting this together, we get
 \[\mathbb E(|\eta-m|)\geq \max\big(|\lambda-m|, \mathbb E(|\eta-\lambda|-|\lambda-m|)\big)\geq \max\bigg(|\lambda-m|, \frac{2}{e^2}\sqrt{\lambda}-|\lambda-m|\bigg)\geq \frac{\sqrt{\lambda}}{e^2}.\]  
Now we apply this for the conditional distribution of $\mathcal E^{\circ}$ with $m=\lfloor \lfloor cn^2\rfloor/2\rfloor$. We obtain 
\begin{align*}\mathbb E(|\mathcal E^{\circ}-m|)&=\mathbb E(\mathbb E(|\mathcal E^{\circ}-m|)|\mathcal G)\geq \frac{1}{e^2}\mathbb E\bigg(\sqrt{c\sum_{1\leq i<j\leq n}\xi_i \xi_j+\frac c2\sum_{i=1}^n \xi_i^2}\bigg)\\&=\frac{1}{e^2}\mathbb E\bigg(\sqrt{\frac c2}\sum_{i=1}^n \xi_i\bigg)=\frac{1}{e^2}\sqrt{\frac c2}\cdot n.\end{align*}
Therefore, since $m$ is the number of edges in the PAG model, we conclude that for every coupling of $\mathbb G_{\rm PAG}$ and $\mathbb G_W^{\circ}$, we have 
\[d_{\boxtimes}(\mathbb G_{\rm PAG}, \mathbb G_W^{\circ})\geq \frac{1}{e^2}\sqrt{\frac c2}\cdot\frac 1n. \]

\begin{remark}
In this paper we considered the jumble distance between the two random models for the dense PAG graph, as that is the more natural distance notion for multigraphs generated by unbounded graphons (in this particular case, this corresponds to the unboundedness of the parameters of the Poisson distributions). However, as each finite multigraph generated is bounded per se, one may wonder if it is possible to say anything about the cut distance between, e.g., $\mb{G}_{\mr{PAG}}$ and $\mb{G}^\circ_W$.\\
We recall that the cut distance of two graphs on the same set of $n$ vertices is defined as
\[d_{\square}(G, H)=\frac 1{n^2}\cdot\max_{S,T}\bigg|\sum_{i\in S, j\in T} U_{ij}-V_{ij}\bigg|.\]
It is easily seen that $d_\square\leq d_\boxtimes$, hence the upper bounds given for the jumble distance apply a fortiori to the cut distance as well. On the other hand, the methods used in this paper do not yield stronger bounds for the cut norm distance.
\end{remark}

\section*{Acknowledgements}
The first author was supported by the Hungarian National Research, Development and Innovation Office, NKFIH grant $\mathrm{n}^\circ$ K108615 and by the MTA R\'enyi Institute Lend\"ulet Limits of Structures Research Group. 
The second author has received funding from the European Research Council under the European Union's Seventh Framework Programme (FP7/2007-2013) / ERC grant agreement $\mathrm{n}^\circ$617747, and from the MTA R\'enyi Institute Lend\"ulet Limits of Structures Research Group.

\end{document}